\documentclass{amsart}
\usepackage[utf8]{inputenc}
\usepackage[T1]{fontenc}
\usepackage{amssymb,amsthm,amsmath,graphicx, longtable, wrapfig, rotating, microtype, xparse, enumitem, mathtools, caption, subcaption, chngcntr, xcolor, bookmark}
\usepackage{hyperref}

\theoremstyle{plain} 
\theoremstyle{plain}
\newtheorem{theorem}{Theorem}
\newtheorem{lemma}[theorem]{Lemma}
\newtheorem{corollary}[theorem]{Corollary}

\newtheorem{proposition}[theorem]{Proposition}
\newtheorem*{definition}{Definition}

\newcommand{\om}{\omega}
\newcommand{\lex}{\mathrm{lex}}
\newcommand{\al}{{\alpha}}

\newcommand\Nat{\mathbf{N}}

\newcommand\RCA{\mathsf{RCA}}
\newcommand\ACA{\mathsf{ACA}}

\newcommand\ATR{\mathsf{ATR}}

\newcommand\RT{\mathsf{RT}}

\newcommand\FreeS{\mathsf{FS}}
\newcommand\TS{\mathsf{TS}}

\newcommand\RRT{\mathsf{RRT}}

\newcommand\FS{\mathsf{FS}}

\newcommand\sW{\mathrm{sW}}

\newcommand\WO{\mathrm{WO}}

\newcommand\KO{\mathcal{O}}

\newcommand\Q{\mathsf{Q}}
\newcommand\PP{\mathsf{P}}

\newcommand\conc{^\smallfrown}

\title[Free sets, thin sets and rainbows for barriers]{\texorpdfstring{Free sets, thin sets and rainbows for barriers}{Free sets, thin sets and rainbows for barriers}}
\author[L. Carlucci]{Lorenzo Carlucci}
\author[O. Gjetaj]{Oriola Gjetaj}

\begin{document}

\begin{abstract}
We formulate and prove the generalizations of Friedman's free set and thin set theorems and of the rainbow Ramsey theorem to colorings of barriers. We analyze the strength of these theorems from the point of view of computability theory proving some upper and lower bounds on the complexity of solutions for computable instances and some uniform computable reductions. 
We obtain as corollaries some proof-theoretical results on the logical strength of the theorems, in the spirit of reverse mathematics.

\end{abstract}

\maketitle

\section{Introduction and motivation}

Many variants or consequences of Ramsey’s theorem have been extensively investigated in computability, proof theory and reverse mathematics. Among these, the free set theorem, the thin set theorem and the rainbow Ramsey theorem have attracted significant interest in recent decades (see, e.g., \cite{Cholak_Giusto_Hirst_Jockusch_2005,wang2014some,Pat:16,cholak2020thin,Liu2022-LIUTRM}). 

The free set theorem ($\FS^n$) states that, for every coloring of the $n$-subsets of the natural numbers in unboundedly many colors, there exists an infinite set $H\subseteq \Nat$ such that for all $n$-subsets $s$ of $H$, the color of $s$ is either not in $H$ or else is in $s$ itself. Such a set is called {\em free} for the coloring. The thin set theorem, denoted $\TS^n_\om$ (or just $\TS^n$) is a weak variant of the free set theorem asserting that for any coloring of the $n$-subsets of $\Nat$ there is an infinite set $H\subseteq\Nat$ such that the $n$-subsets of $H$ avoid at least one color. Such a set is called {\em thin} for the coloring.
Both the free set theorem and the thin set theorem are consequences of Ramsey's theorem.  Cholak et al. \cite{Cholak_Giusto_Hirst_Jockusch_2005} started the study of the logico-computational content of these theorems. Contrary to Ramsey's theorem, for any $n\geq 3$, none of these principles for colorings of $n$-subsets codes the halting set. This surprising result was established by Wang \cite{wang2014some}.

The rainbow Ramsey theorem ($\RRT^n_k$) is an “anti-Ramsey” principle that guarantees, for any coloring of the $n$-subsets of the natural numbers such that no color is used more than $k$ times, the existence of an infinite set $H\subseteq\Nat$ whose $n$-subsets are assigned distinct colors (i.e., the coloring is injective on the $n$-subsets). Such a set is called a {\em rainbow} for the coloring. Galvin \cite{galvinBorelSetsRamsey1973} observed that the rainbow Ramsey theorem follows from the classical Ramsey's theorem. This theorem is studied by Csima and Mileti \cite{CsimaMileti_2009} and Wang \cite{wang2014some} from the perspective of computability and reverse mathematics. The rainbow Ramsey theorem for pairs ($\RRT^2_2$) is strictly weaker than Ramsey's theorem for pairs ($\RT^2_2$) over $\RCA_0$ and for no finite dimension $n$ and no bound $k$ the rainbow Ramsey theorem for $k$-bounded colorings of $n$-sets ($\RRT^n_k$) does code the halting set \cite{wang2014some}.

In joint work with Le Hou\'erou and Patey \cite{carlucci2024ramseyliketheoremsschreierbarrier} we generalized these principles to colorings of objects of unbounded dimensions, in particular to colorings of sets $s\subseteq \Nat$ such that $|s|=\min(s)+1$. These sets are known as {\em exactly large sets} or {\em $\om$-size} sets in logic literature (see \cite{carlucci2014strength,carlucci_volpi_zdanowski_25,marcone_montalban_volpi_25}). 

Ramsey's theorem for colorings of exactly large sets is computationally and proof-theoretically stronger than the usual Ramsey theorem for fixed finite dimensions ($\RT^n_k$) and even stronger than Ramsey's theorem for all finite dimensions ($\forall n \RT^n_k$). From the point of view of computability, the theorem corresponds to (closure under) the $\om$-th Turing jump (see \cite{Clote1984} and \cite{carlucci2014strength}). 
Generalizations to exactly large sets of the free set, thin set and rainbow Ramsey theorems were studied in \cite{carlucci2024ramseyliketheoremsschreierbarrier}. The free set and thin set theorems for exactly $\om$-large sets turn out to be significantly stronger than their versions for sets of fixed finite size. In particular, they can code the $\om$-th Turing jump and imply the system $\ACA_0^{+}$. By contrast, the rainbow Ramsey theorem for exactly $\om$-large sets satisfies cone avoidance and therefore does {\em not} code the first Turing jump. 

The family of exactly large sets forms a barrier in the sense of Nash-Williams better quasi-ordering theory and is known as the Schreier barrier in Ramsey theory and Banach space theory literature (see, e.g., \cite{Todorcevic+2010}). Indeed, Nash-Williams proved that Ramsey's theorem holds for arbitrary barriers. The computational strength of Ramsey's theorem for barriers is investigated in \cite{Clote1984}. It is thus quite natural to inquire into the generalization of the free set, thin set and rainbow Ramsey theorems to general barriers. 

In the present paper we formulate and prove these theorems and establish some bounds on the computational complexity of their solutions. All principles $\PP$ studied in this paper have the form
$\forall X\,(I(X)\to \exists Y\,S(X,Y))$, where $I$ and $S$ are arithmetical formulas. We call such statements $\forall\exists$-principles. An $X$ satisfying $I$ is called an {\em instance} of $\PP$, and a $Y$ satisfying $S(X,Y)$ is called a {\em solution} to $\PP$ for $X$. Our results are naturally stated in terms of strong uniform computable reductions, known as strong Weihrauch reductions (see \cite{Dzhafarov.Mummert2022}).
For $\forall\exists$-principles $\PP$ and $\Q$, we say that $\Q$ is strongly Weihrauch reducible to $\PP$,
written $\Q \leq_{\sW} \PP$, if there are Turing functionals $\Phi$ and $\Psi$ such that for every instance $X$ of
$\Q$, $\Phi(X)$ is an instance of $\PP$, and every solution $\hat{Y}$ to $\Phi(X)$ yields a solution $\Psi(\hat{Y})$ to $\Q$ for $X$.
We occasionally also mention results in reverse mathematics. We refer to \cite{Dzhafarov.Mummert2022} for background and definitions.

The paper is organized as follows. In Section \ref{sec:RT} we introduce barriers and Nash-Williams' Ramsey theorem for barriers. In Section \ref{sec:up} we introduce and prove the free set, thin set and rainbow Ramsey theorems for barriers. In Section \ref{sec:low} we prove some lower bounds on their solutions. 

\section{Ramsey theory on barriers}\label{sec:RT}

We start by fixing some notation. We use $\Nat$ to denote the set of natural numbers and $\Nat^+$ to denote the set of positive integers. We use uppercase letters $X, Y, Z$ to denote subsets of $\Nat$ which may be finite or infinite. We use lowercase letters $s, t$ to denote finite subsets of $\Nat$. 
If $X \subseteq \Nat$ we denote by $[X]^{< \om}$ (respectively $[X]^{\om}$) the set of finite (respectively, infinite) subsets of $X$. For $n \in \Nat$ we denote by $[X]^n$ the set of subsets of $X$ with exactly $n$ elements (which we call the $n$-subsets of $X$).
We identify a subset $X$ of $\Nat$ with the strictly increasing sequence (finite or infinite) which enumerates it and we denote by $X(i)$ or $X_i$ the element in position $i$ in the sequence. If $s$ is such a finite sequence then $|s|$ denotes its length (which coincides with the cardinality of the range of $s$ since $s$ is increasing) and we write $s = (s_0, \ldots, s_{|s|-1})$. 
Given $s, X \subseteq \Nat$, by $s \sqsubseteq X$ we mean that $s$, as a sequence, is an initial segment (or prefix) of $X$. This is stronger than $s \subseteq X$, which denotes set-theoretic inclusion as usual. The irreflexive versions of the previous relations are denoted by $\sqsubset$ and $\subset$.

By $s < X$ we mean that each element of $s$ is strictly smaller than each element of $X$ or equivalently (when $s$ and $X$ are nonempty) $\max s < \min X$. We write $s \conc X$ for the concatenation of the sequences $s$ and $X$. Notice that in our setting $s \conc X$ does make sense only if $s < X$ and, set-theoretically, coincides with $s \cup X$. Let $F\subseteq [\Nat]^{<\om}$ and $X \in [\Nat]^\om$ then $ F\vert X = \{ s \in F \,:\, s \subseteq X\}$. Thus, $F\vert X$ is an abbreviation for $[X]^{<\om}\cap F$. Let $X \in [\Nat]^\om$ and $n\in\Nat$, then we denote by $X/n$ (the $n$-tail of $X$) the set $X \setminus [0,n]$. 

We now introduce the basic definitions concerning barriers. Barriers are combinatorial objects first introduced by Nash-Williams \cite{nash-williams1968better} to define better quasi orderings. 

\begin{definition}[Barriers] \label{barrierDef}
A set $B\subseteq [\Nat]^{<\om}$ is a barrier on $X\in [\Nat]^\om$ if the following conditions are met:
\begin{enumerate}
\item (Base) $\bigcup B = X$;
\item (Sperner property) For all $s, t \in B$, $s \not\subset t$;
\item (Density) For all $Y\in [X]^\om$ there exists an $s\in B$ such that $s\sqsubset Y$. 
\end{enumerate}
\end{definition}
The set $X = \bigcup B$ is called the base of $B$ and is denoted by $b(B)$. The simplest examples of barriers are the families $[\Nat]^n$ for $n \in \Nat$. A trivial barrier often refers to the barrier of singletons, $[\Nat]^1 = \{\{ n \}  :  n \in \Nat\}$.

The family $\{ s \subset \Nat : |s| = \min(s)+1\}$ is a barrier, known as the Schreier barrier or as the barrier of exactly $\om$-large sets (also called $\om$-size sets). For $X\subseteq \Nat$ the family $\{ s \subset X  :  |s| = \min(s)+1\}$ is known as the Schreier barrier on $X$ \cite{Todorcevic+2010}.

Marcone~\cite{marcone2005wqo} started the study of barriers in the context of reverse mathematics. Among other things he showed that $\RCA_0$ proves that if $B$ is a barrier then $b(B)$ exists as a set and $B$ is isomorphic to a barrier $B'$ with $b(B') =\Nat$.

A fundamental fact, proved by Pouzet \cite{pouzet1972premeilleurordres}, is that the lexicographic order on barriers is always a well-ordering. Given a barrier $B \subseteq [\Nat]^{<\om}$ and $s, t \in B$, let $s <_\lex t$ if $s(i) < t(i)$ for the least~$i\in\Nat$ such that $s(i) \neq t(i)$, if it exists. Based on this, the order type of a barrier is the order type of its lexicographic order. If $B$ is a barrier we denote its order type by $ot(B)$. Assous~\cite{Assous} characterized the order types of barriers, and proved that they are either of the form $\om^n$ for some~$n \in \Nat^+$, or $\om^\alpha \cdot k$ for some~$\alpha \geq \om$ and $k \in \Nat^+$. The Schreier barrier is an example of a barrier of order type $\om^\om$. 

If $B$ is a barrier on $X\in [\Nat]^\om$, then for any infinite subset $Y \in [X]^\om$, the set $B\vert Y = \{s \in B~\vert~s \subseteq Y \}$ forms a barrier on $Y$, which we call a sub-barrier of $B$. 
Additionally, if $B$ is a non-trivial barrier on $X$ and $n\in X$ then $B_n =\{s \in [\Nat/n]^{<\om}  :  \{n\}\cup s \in B\}$ is a barrier on $X/n$. 
Note that $B_n$ is not in general a sub-barrier of $B$.

\begin{proposition}\label{prop:sub_bar}
Let $B$ be a barrier on $X\in [\Nat]^\om$. Then $B' \subseteq B$ is a barrier if and only if there exists $Y \in [X]^\om$ such that
$B\vert Y = B'$.
\end{proposition}

\begin{proof}
See the proof of Proposition 2.9 in \cite{carroy_pequignot_2014}. Note that both directions between $B'$ and $Y$ are computable.
\end{proof}

One reason for interest in barriers is that Ramsey's theorem holds for finite colorings of barriers as the following well-known result by Nash-Williams \cite{nash-williams1968better} shows.

\begin{theorem}[Barrier Ramsey Theorem\footnote{Note that in the recent \cite{marcone_montalban_volpi_25} the “barrier Ramsey theorem” denotes a finitary theorem generalizing the Paris-Harrington principle.}]
Fix a barrier $B$ on $X\in [\Nat]^\omega$. For every $f: B \to k$, where $k$ is a positive integer, there exists an infinite set $H\in [X]^\omega$ such that $f$ is constant on $B\vert H$. We say that $H$ is monochromatic for $f$. We denote this principle by $\RT^B_k$.
\end{theorem}

More succinctly, in the light of Proposition \ref{prop:sub_bar}, the previous theorem asserts that for all barriers $B$ and for all $k$-coloring $f$ of $B$ there exists a sub-barrier $B'\subseteq B$ such that
$f$ is constant on $B'$. We call $B'$ monochromatic for $f$.

Ramsey's theorem for the Schreier barrier, studied in \cite{carlucci2014strength}, is a particular case of Nash-Williams' barrier Ramsey theorem, which in turn is a consequence of the clopen Ramsey theorem (see \cite{SIM:SOSOA}). 
The barrier Ramsey theorem {\em for all barriers} is equivalent to the system $\ATR_0$ (see Theorem 4.9 in \cite{marcone2005wqo}). Thus, the barrier Ramsey theorem is proof-theoretically much stronger than Ramsey's theorem for the Schreier barrier, which stands at the level of $\ACA_0^+$ as proved in \cite{carlucci2014strength}.

Clote~\cite{Clote1984} showed that the order-type of barriers is relevant to the computability-theoretic complexity of the barrier Ramsey theorem. He obtained a level-wise analysis of the complexity of solutions to computable instances of $\RT^B_k$ in the hyperarithmetic hierarchy as a function of the order type of the barrier $B$. He proved in particular the following theorem. The $\Sigma^0_\al$ sets are defined as in \cite{Rogers1987}.

\begin{theorem}[Clote~\cite{Clote1984}]\label{thm:Clote}
Fix~$k \in \Nat^+$.
\begin{enumerate}
    \item For every computable ordinal $\alpha$, for every computable barrier $B$ of order type at most $\om^{1+\al}$ and every computable coloring $f : B \to k$, there is an infinite $f$-monochromatic set computable in~$\emptyset^{(\al)}$.
    \item For every computable limit ordinal $\alpha$ there exists a computable barrier $B$ of order type $\om^{\al}$ and a computable coloring $f : B \to 2$ such that every infinite $f$-monochromatic set computes~$\emptyset^{( \al)}$.
    \item For every computable ordinal $\alpha$ there exists a computable barrier $B$ of order type $\om^{\al}$ and a computable coloring $f : B \to 2$ that omits all $\Sigma_{\om^\al}^* = \bigcup_{\beta < \om^\alpha}\Sigma^0_\beta$-sets.
\end{enumerate}
\end{theorem}

\section{Free sets, thin sets and rainbows for barriers} \label{sec:up}

In this section we introduce and prove the natural generalizations of the free set, thin set and rainbow Ramsey theorems to arbitrary barriers.

\begin{definition}[Barrier Free Set Theorem]
Fix a barrier $B$ on $X\in [\Nat]^\omega$. For every $f:B \to \Nat$ there exists a barrier $B'\subseteq B$ such that for all $s \in B'$ if $f(s) \in b(B')$ then $f(s) \in s$. We say that $B'$ is free for $f$. If $B'=B\vert Y$ for some $Y\in [X]^\om$, we call $Y$ a free set for $f$.
We denote this principle by $\FreeS^B$. 
\end{definition}

\begin{definition}[Barrier Thin Set Theorem]
Fix a barrier $B$ on $X\in [\Nat]^\omega$. For every $f:B \to \Nat$ there exists a barrier $B'\subseteq B$ such that the image of $f$ restricted to $B'$ is different from $\Nat$. We say that $B'$ is thin for $f$. If $B'=B\vert Y$ for some $Y\in [X]^\om$, we call $Y$ a thin set for $f$.
We denote this principle by $\TS^B$. 
\end{definition}

Let $k\in\Nat$. A coloring $f:X \to Y$ is called $k$-bounded if for
all $x \in Y$ $|f^{-1}(x)|\leq k$.

\begin{definition}[Barrier Rainbow Ramsey Theorem]
Fix a barrier $B$ on $X\in [\Nat]^\omega$ and let $k\in \Nat$.  
For every $k$-bounded $f:B \to \Nat$, there exists a barrier $B' \subseteq B$ such that $f$ is injective on $B'$. We say that $B'$ is a rainbow for $f$. If $B'=B\vert Y$ for some $Y\in [X]^\om$, we call $Y$ a rainbow for $f$. We denote this principle by $\RRT^B_k$.
\end{definition}

The proofs of the thin set and rainbow Ramsey theorems for barriers are straightforward generalizations of the proofs for the finite dimensional version of these theorems, see sections \ref{sub:thin} and \ref{sub:rain}. This observation was indeed the initial motivation for the present project. The case of the free set theorem for barriers is more interesting. We treat this case in the next subsection.

\subsection{Barrier free set theorem}\label{sub:free}

We prove the barrier free set theorem from the barrier Ramsey theorem by generalizing the proof for the Schreier barrier versions of these theorems from \cite{carlucci2024ramseyliketheoremsschreierbarrier}.
Interestingly, some operations that {\em prima facie} looked specific to exactly $\om$-large sets naturally generalize to arbitrary barriers. Let $X \in [\Nat]^\om$ and let $S_X$ be the Schreier barrier on $X$. The following facts are easy to check. 

(1) If $s \in S_X$ and $k \in X$ and $k < \min(s)$
then $\{ k\} \cup \{ s_0, s_1, \dots, s_{k-1}\} \in S_X$.

(2) If $s = \{s_0, \dots, s_{s_0}\} \in S_X$ and $k \in X$ and $s_i < k < s_{i+1}$ for some $i \in [0, s_0-1]$ then $\{ s_0, s_1, \dots, s_{i}, k, s_{i+1}, \dots s_{s_0-1}\} \in S_X$. 

We jointly generalize the above two properties to general barriers in the following lemma. 

\begin{lemma}
    Let $B$ be a barrier with base $X\in [\Nat]^\om$ and let $s = \{s_0,\ldots, s_n\}\in B$ with $s_0<\ldots < s_n$. Fix $i\in [-1,n-1]$ (with $s_{-1}=-1$) and let $k \in (s_{i}, s_{i+1})$ such that $k \in X$.
    Then there exists a (unique) element of $B$ of the form
    $$\{s_0,\ldots, s_i,k, s_{i+1}, \dots, s_j\}$$ 
    with $j < n$. We call such an element the $k$-variant of $s$ and denote it by $s[k]$.
\end{lemma}

\begin{proof} 
Such an element must exist by the Density and Sperner properties by considering the infinite set
$$\{s_0, \dots, s_i\} \cup \{k\} \cup \{s_{i+1}, \dots, s_n\}\cup X/s_n.$$ 

Necessarily the desired initial segment must stop not earlier than $\{s_0, \dots, s_i\}\cup\{k\}$ (otherwise it would be 
a subset of $s$) and not later than $\{s_0, \dots, s_i\}\cup\{k\} \cup \{s_{i+1}, \dots, s_{n-1}\}$ (otherwise $s$ would be a subset of it).
 
\end{proof}

Note that, whenever defined, $s[k] <_\lex s$. If $s \in B$ and $k \in X$ is such that $s_{n-1} < k < s_n$, then $(s_0, \dots, s_{n-1}, k) \in B$. For $k < \min(s)$ the notion of $k$-variant is closely related with the definition of the relation $\lhd$ on a barrier that is a key step in the definition of a better quasi order. For $s, t \in [\Nat]^{<\om}$ the relation $s \lhd t$ is defined to hold if there exists a $u\in [\Nat]^{<\om}$ such that $s \sqsubseteq u$ and $t \sqsubseteq u[1, |u|]$. If $s$ is the $k$-variant of $t$ for some $k < \min(t)$ then $s \lhd t$ holds. If $s$ and $t$ are members of a barrier and $s \lhd t$, then  the reverse implication also holds; $s$ is the $k$-variant of $t$ for some $k < \min(t)$.

We next observe that barriers are preserved under translations. We use the following notation. If $X = \{ x_1, x_2, \ldots\}\subseteq \Nat$ we define $X^+=\{x_i+1  :  i \in \Nat\}$.

\begin{lemma}
    Let $B$ be a barrier with base $X\in [\Nat]^\om$. Let 
    $$B^+ = \{ s^+ \cup\{m\}  :  s \in B \text{ and } m \in X^+/(\max(s)+1)\}$$ 
     where $s^+ = \{x+1: x\in s\}$. $B^+$ is barrier with base $X^+$.
\end{lemma}
\begin{proof}
    We verify that $B^+$ satisfies the Base, Sperner and Density properties:
    
(Base) We need to show that the base of $B^+$ is infinite and is equal to $X^+$. Notice that every element of $s'\in B^+$ is either $s_i+1$ with $s_i\in X$, or the last coordinate $m\in X^+$. Thus, $b(B^+)\subseteq X^+$. Conversely, let $x \in X^+$. Then $x= x_k+1$ for some $x_k \in X$. Since $X$ is the base of $B$ there is some $s\in B$ such that $x_k \in s$. Let $m \in X^+/\max(s^+)$. Then $s^+ \cup \{m\} \in B^+$ and contains $x_k+1$. 

(Sperner property) For all $s', t' \in B^+$ such that $s'\neq t'$, $s' \not\subseteq t'$. Let $s' = s^+ \cup \{ a\}$ and 
$t' = t^+ \cup \{ b\}$ for some $s, t \in B$, $a > \max(s)+1$ and $b > \max(t) +1$ and such that $s' \neq t'$. Obviously if 
$s^+ = t^+$ then $a \neq b$ and therefore  $s^+ \cup \{ a\} \not\subseteq t^+ \cup \{ b\}$. So suppose that $s^+ \neq t^+$, i.e.
$s \neq t$. By the Sperner Property on $B$ we have $s \not\subseteq t$. Suppose $s^+ \cup \{ a\} \subseteq t^+ \cup \{ b\}$.
If $s^+ \cup \{ a \} \subseteq t^+$ then $s \subseteq t$, which is impossible. 

(Density) We want to show that for all $Y \in [X^+]^\om$ there exists an $s' \in B^+$ such that $s'\sqsubset Y$.
By definition we have that $Y = Z^+$ for some $Z \in [X]^\om$. By the density of $B$ on $X$
there exists an $s \in B$ such that $s \sqsubset Z$. Thus $s^+ \sqsubset Z^+$. Let $z$ be the first 
element of $Z$ larger than $\max(s^+)$. Then $s^+ \cup \{z\} \sqsubset Z^+$ and $s^+ \cup \{z\}\in B^+$. 
 \end{proof}

In the next theorem we prove the free set theorem for barriers using Ramsey's theorem for barriers. We express this result (and later positive results) in terms of strong Weihrauch reductions, which we denote by $\leq_\sW$ (see \cite{Dzhafarov.Mummert2022} for definitions). 

\begin{theorem}\label{thm:fs_red_rt}
For each barrier $B$ with base $\Nat$, $\FS^B\leq_{\sW} \RT^{B^+}_2$. 
\end{theorem}

\begin{proof}
Let $B$ be a barrier with base $\Nat$ and let $f:B\to \Nat$ be an instance of $\FS^B$.
For $s = (s_0, \dots, s_n) \in B^+$, let $s\ominus 1 = (s_0 - 1, \dots, s_{n-1}-1)$. Then $s \ominus 1 \in B$. Define $g$ by recursion on $<_{\lex}$ as follows. 

\begin{equation*}
    g(s) = \begin{cases}
           0 \hspace{1.5cm} & \text{if } f(s\ominus 1) = s_i-1 \textit{ for some } i  < n,  \\
           1 - g(s[f(s\ominus 1)+1]) & \text{if } f(s\ominus 1) < s_0 - 1, \text{or, for some } i < n-1,\\
           \hspace{1.2cm} & f(s\ominus 1) \in (s_i-1, s_{i+1}-1),\\
           0 \hspace{1.5cm} & \text{if } f(s\ominus 1) \in (s_{n-1}-1, s_n-1), \\
           1 \hspace{1.5cm} & \text{otherwise } (\textit{if } f(s\ominus 1) \geq s_n)   
           \end{cases} \quad
\end{equation*}

Note that $g$ is well-defined: if $f(s \ominus 1) < s_0 - 1$ then $f(s \ominus 1) + 1 < s_0 = \min(s)$, so that 
$s[f(s\ominus 1)+1]$ is well-defined, is in $B^+$ and is $<_\lex$-smaller than $s$. If $f(s\ominus 1) \in (s_i-1, s_{i+1}-1)$ for some  $i < n - 1$, then $f(s\ominus 1) +1 \in (s_i, s_{i+1})$ for some $i< n-1$ and $s[f(s\ominus 1)+1]$ is defined, is in $B^+$ and is $<_\lex$-smaller than $s$. 
Let $H = \{x_0, x_1, \dots \}\subseteq \Nat^+$ be infinite such that $B^+ | H$ is $g$-monochromatic as given by $\RT^{B^+}_2$. We claim that $B\vert H^-$ is $f$-free, where $H^- = \{x_0 - 1, x_1 - 1, \dots \}$. 
Consider some $(s_0 - 1, \dots, s_{n - 1} - 1) \in B\vert H^-$. Then $\{s_0, \dots, s_{n - 1}\} \subseteq H$ and for every $k \in H/s_{n-1}$, $(s_0, \dots, s_{n-1}, k) \in B^+\vert H$. 
To show that $B\vert H^-$ is $f$-free it is sufficient to show that for some $k$, writing $s=(s_0,\ldots,s_{n-1}, k) \in B^+\vert H$, we have that if $f(s\ominus 1) \in H^-$ then $f(s\ominus 1) \in s\ominus 1$, since all elements of $B\vert H^-$ are of the form $s\ominus 1$ for some $s \in B^+\vert H$. 
There are two cases:

\textbf{Case 1:} $B^+\vert H$ is monochromatic for the color $0$. Take $s_n$ to be the next element of $H$ after $s_{n - 1}$ and write $s = (s_0, \dots, s_{n})$, then $g(s) = 0$. We consider three subcases:

\textbf{Subcase 1.1:} $f(s \ominus 1) = s_i - 1$ for some $i < n$. 
Then $f(s\ominus 1)\in s\ominus 1$, so we are done.

\textbf{Subcase 1.2:} $f(s \ominus 1) \in (s_{n - 1} - 1, s_{n} - 1)$. 
Then $f(s\ominus 1)\notin H^-$ by definition of $s_n$ as the next element of $H$ after $s_{n-1}$.

\textbf{Subcase 1.3:} $f(s \ominus 1) < s_0 - 1$ or  $f(s \ominus 1) \in (s_i - 1, s_{i+1} - 1)$ for some $i < s_0 - 1$ and $g(s[f(s \ominus 1)+1]) = 1$. 
Since $f(s\ominus 1)+1 \in H$, we have $s[f(s\ominus 1)+1] \in B^+\vert H$. But then 
$g(s[f(s \ominus 1)+1]) = 1$ contradicts the fact that $B^+\vert H$ is $g$-monochromatic for the color $0$. 

\textbf{Case 2:} $B^+\vert H$ is monochromatic for the color $1$. Take $s_n \in H$ bigger than $s_{n - 1}$ and write $s = (s_0, \dots, s_{n})$, then $s \in B^+\vert H$ and $g(s) = 1$. We consider two subcases:

\textbf{Subcase 2.1:} $f(s \ominus 1) \geq s_{n}$. This case is impossible, indeed, as $H$ is infinite, there exists some element $x \in H$ such that $x > f(s \ominus 1) + 1$ and therefore $f(s \ominus 1) \in (s_{n} - 1, x - 1)$, which leads to $g(s_0, \dots, s_{n - 1}, x) = 0$ contradicting the fact that $B^+\vert H$ is $g$-monochromatic for the color $1$. Note that by definition of $B^+$, $(s_0, \dots, s_{n - 1}, x)\in B^+$. 

\textbf{Subcase 2.2:} $f(s \ominus 1) < s_0 - 1$ or $f(s \ominus 1) \in (s_i - 1, s_{i+1} - 1)$ for some $i < n - 1$ and $g(s[f(s \ominus 1)+1]) = 0$.  
Since $f(s\ominus 1)+1 \in H$, we have that $s[f(s\ominus )+1] \in B^+\vert H$. But 
$g(s[f(s \ominus 1)+1]) = 0$ contradicts the fact that $B^+\vert H$ is $g$-monochromatic for the color $1$. 

Therefore $B\vert H^-$ is $f$-free.
Notice that $g$ is uniformly computable in $f$ and that $H^-$ is uniformly computable in $H$. Hence, $\FS^B \leq_{\sW} \RT^{B^+}_2$.
 \end{proof}

In terms of reverse mathematics, the above proof shows that $\FS^B$ follows from $\RT^{B^+}_2$ over $\RCA_0 + \WO(ot(B^+))$, where $ot(B^+)$ is the order type of $B^+$. In terms of computability $\FS^B$ inherits the upper bound on solutions to $\RT^{B^+}_2$ from Theorem \ref{thm:Clote}. 
It is easy to show that, for every barrier $B$,
$$ot(B^+)= \omega\cdot ot(B),$$
since decomposes $B^+$ into intervals $I_s:=(s^+,m)$ with $m>\max(s^+)$ indexed by $s\in B$ such that each interval is ordered according $m$ (so that $ot(I_s) =\om$) and the relative order of the intervals is determined by the order of their index $s$.
Since $\RCA_0$ proves $\forall \alpha(\WO(\alpha) \to \WO(\omega\cdot\alpha))$, the proof of Theorem \ref{thm:fs_red_rt} shows that $\RT^{B^+}_2$ implies $\FS^B$ over $\WO(ot(B))$. 

\subsection{Barrier thin set theorem}\label{sub:thin}

The proof of the following proposition is a straightforward adaptation from the analogous result for the finite dimensional versions of the thin set theorem (Proposition 5.2 in \cite{Dorais.etal2015}).

\begin{proposition}
For all barriers $B$, $\TS^B \leq_{\sW} \RT^B_2.$
\end{proposition}

\begin{proof}
 Let $f: B \to \Nat$ be an instance of $\TS^B$. Define $g:B\to 2$ as follows:

$$
g(s) = \begin{cases}
f(s) & \mbox{ if } f(s) = 0,\\
1 & \mbox{ otherwise.}
\end{cases}
$$
Any $g$-monochromatic sub-barrier $B'$ of $B$ is obviously thin for $f$.
 \end{proof}

The barrier thin set theorem also follows from the barrier free set theorem. 
The proof of the next theorem is the same as the proof for the finite dimensional case (Theorem 3.2 in \cite{Cholak_Giusto_Hirst_Jockusch_2005}). The observation that the argument works for barriers started the investigations on $\FS$ and $\TS$ for barriers.

\begin{theorem}
For all barriers $B$, $ \TS^B \leq_{\sW}\FreeS^B $.
\end{theorem}

\begin{proof}
Let $f:B\to \Nat$. Let $X$ be an infinite set such that $X \subseteq b(B)$ and $B\vert X$ is free for $f$.
Let $Y$ be a non-empty subset of $X$ such that $X\setminus Y$ is infinite. 
We claim that $B\vert(X\setminus Y)$ is thin for $f$. 
Assume, by way of contradiction, that for all $n \in \Nat$ there exists an $s_n\in B\vert (X\setminus Y)$ such that $f(s_n) = n$. Take $n\in Y$. Thus, $n \in X$. Since $B\vert X$ is free for $f$, it must be the case that $n \in s_n$, contradicting the fact that $s_n$ was chosen in $X\setminus Y$. 
 \end{proof}

As a corollary of the above reductions and Theorem \ref{thm:Clote} we get the following.

\begin{corollary}
For every computable ordinal $\alpha$, for every computable barrier $B$ of order type at most $\om^{1+\al}$ and every computable coloring $f : B \to \Nat$, there is an infinite $f$-thin set computable in~$\emptyset^{(\al)}$.
\end{corollary}

\subsection{Barrier rainbow Ramsey theorem}\label{sub:rain}

Galvin's proof of rainbow Ramsey theorem from Ramsey's theorem (see the proof of \cite[Theorem 1.6]{CsimaMileti_2009}) adapts with minor modifications to the barrier version.

\begin{proposition}
For all $k\in\Nat^+$, for all barriers $B$, $\RRT^B_k \leq_{\sW} \RT^B_k$.
\end{proposition}

\begin{proof}
Let $f:B \to \Nat$ be $k$-bounded. Fix a $B$-computable bijection $c : B \to \Nat$. 
Define $g:B \to k$ as follows:
$$ g(s) = |\{ t \in B : c(t) < c(s) \text{ and } f(s) = f(t)\}|.$$
The set $\{ t \in B : c(t) < c(s)\}$ is finite and finite coded.
The fact that $g$ is a $k$-coloring follows from the hypothesis that $f$ is $k$-bounded.

Let $H\subseteq\Nat$ be an infinite subset of $b(B)$ such that $g$ is constant on $B\vert H$, as given by $\RT^B_k$. We claim that $B\vert H$ is a rainbow for $f$.

Let $s, t \in B\vert H$. Since $g(s) = g(t)$ and either $c(s) < c(t)$ or $c(t) < c(s)$, we have that $f(s) \neq f(t)$. Thus $B\vert H$ is a rainbow for $f$. 
 \end{proof}

As is the case for the Schreier barrier (see \cite[Proposition 3.11]{carlucci2024ramseyliketheoremsschreierbarrier}) and for $[\Nat]^n$ (see \cite[Theorem 4.2]{wang2014some}), the rainbow Ramsey theorem for $2$-bounded colorings of barriers follows from the free set theorem for
barriers. We don't know if the same holds for $k>2$.

\begin{proposition}
For all barriers $B$, $\RRT^{B}_2 \leq_{\sW} \FreeS^{B} $.
\end{proposition}

\begin{proof}
Let $B$ be a barrier. 
Fix a $B$-computable bijection $c: B \to \Nat$ so that for each $s$ the set $\{t\in B : c(t)<c(s)\}$ is finite and finite coded.
Let $f:B\to\Nat$ be $2$-bounded. For $s \in B$ define $g(s)$ as follows:
$$
g(s)=
\begin{cases}
\min(t\setminus s) & \mbox{ if there exists a } t\in B \text{ s.t. } c(t)<c(s) \text{ and } f(s)=f(t),\\
0 & \mbox{ otherwise.}
\end{cases}
$$
Since $f$ is $2$-bounded, if $t$ exists in the definition of $g$ then it is unique. If $t$ and $s$ are distinct
elements in the barrier $B$ then $(t\setminus s)\neq \emptyset$, since $t\subseteq s$ is impossible. Let $A$ be 
an infinite subset of the base of $B$ such that $B'=B\vert A$ is $g$-free. We claim that $B'$ is a rainbow for $f$. Suppose otherwise, by way of contradiction, 
as witnessed by $s, t \in B'$ such that $f(s) = f(t)$. Without loss of generality we can assume
$c(t) < c(s)$. Then $g(s) = \min(t\setminus s) \in A \setminus s$, contradicting that $B'$ is $g$-free.
 \end{proof}

As a corollary of the above reductions and Theorem \ref{thm:Clote} we get the following.

\begin{corollary}
For every computable ordinal $\alpha$, for every computable barrier $B$ of order type at most $\om^{1+\al}$ and every computable coloring $f : B \to \Nat$, there is an infinite $f$-rainbow computable in~$\emptyset^{(\al)}$.
\end{corollary}

\section{Lower bounds}\label{sec:low}

The argument used by Jockusch to establish the seminal result that $\RT^2_2$ omits $\Sigma^0_2$-solutions (Theorem 3.1 in \cite{Jockusch1972}) turned out to be extremely versatile. It stands out as the model whose variations allowed to establish a large number of similar omission results for variants of Ramsey's theorem. Among those are the increasing polarized Ramsey theorem, Hindman's finite sums theorem, the free set theorem, the thin set theorem and the rainbow Ramsey theorem. The fertility of Jockusch's construction is not limited to the lower levels of the arithmetical hierarchy. Indeed, Clote \cite{Clote1984} showed how the scheme of Jockusch's proof can be lifted to obtain weak anti-basis results with respect to the hyperarithmetical hierarchy. In particular Clote used this idea to obtain lower bounds for the barrier Ramsey theorem $\RT^B_2$ as a function of the order-type of the barrier $B$ 
as shown in Theorem \ref{thm:Clote} above.

It is thus natural to inquire whether the same scheme of argument applies to the barrier versions of other variants of Ramsey's theorem for which Jockusch's argument works in the finite case. In this section we verify that this is the case for the barrier thin set theorem and for the barrier rainbow Ramsey theorem. 

\subsection{Clote's canonical barriers and generalized Limit Lemma}

Clote's lifting of Jockusch's argument relies on a generalization of Shoenfield's limit lemma proved in \cite{clote1986generalization}. 
Clote’s proof uses limits taken along barrier elements rather than along the natural numbers. The barriers used in the construction are called canonical barriers, and they are defined using notations for computable ordinals in Kleene’s notation system $\KO$ (see \cite{Rogers1987}).

We need the following definition. For barriers $B$ and $B'$ let $B * B' = \{ s \cup t \,:\, s \in B, t\in B' \mbox{ and }\max(s)<\min(t)\}$. Note that if barriers $B$ and $B'$ have the same base, then $B * B'$ is a barrier on that base.

\begin{definition}[Canonical Barriers]
For each ordinal notation $a \in \mathcal{O}$, the canonical barrier $B_a$ is defined by induction as follows.
\begin{enumerate}
    \item $B_1 = \{( \,)\}$, 
\item $B_2 = \{ (n) \,:\, n \in \Nat\}$, 
\item $B_{2^a} = B_2 * B_a$,
\item $B_{3\cdot 5^z} = \bigcup_{n\in\Nat} \{ (n)\} * B_{\varphi_z(n)}*B_{\varphi_z(n-1)}*\dots*B_{\varphi_z(0)}.$
\end{enumerate}
\end{definition}

Recall that if $3\cdot 5^z$ is a notation then $\varphi_z(0), \varphi_z(1), \varphi_z(2),\dots$ are notations for an increasing sequence of ordinals whose limit is the ordinal denoted by $3\cdot 5^z$. Each $B_a$ has base $\Nat$ and order type $\om^{|a|}$ (see \cite{Clote1984}). Note that Clote's canonical barriers are close (but not identical) to so-called {\em exactly} $\alpha$-large sets (a.k.a. $\alpha$-size sets). 

For an ordinal notation $a \in \mathcal{O}$, limits along a canonical barrier are defined inductively as as follows. Let $g$ be a partial computable function, then
$$ \lim_{s \in B_1} g(x,s) = g(x, ()),$$
$$ \lim_{ s \in B_2} g(x, s) = y \text{ if and only if } \exists m \forall n\geq m~~g(x,(n)) =y,$$
$$ \lim_{s\in B_{2^a}} g(x, s) = \lim_{(n) \in B_2}~ \lim_{\substack{t \in B_{a}\\ n<\min(t)}} g(x, (n)\conc t),$$
$$ \lim_{s\in B_{3\cdot 5^z}} g(x, s) = \lim_{\langle n\rangle\in B_2}~~\lim_{\substack{t_n \in B_{\varphi_z(n)} \\ n<\min(t_n)}} \ldots \lim_{\substack{t_0 \in B_{\varphi_z(0)} \\ \max(t_1)<\min(t_0)}}
g(x, (n) \conc t_n\conc \cdots \conc t_0 ).$$
The limits along a canonical barrier correspond, as noted in \cite{clote1986generalization}, to the notion of limit with respect to a generalized Fréchet filter. 

The proof in \cite{clote1986generalization} of the following generalization of the limit lemma uses effective transfinite induction on ordinal notations. The set $H_a$ for $a\in\KO$ is defined as usual and captures the $a$-iteration of the Turing jump (see, e.g., \cite{Rogers1987} for a definition). By a result of Spector \cite{Spector1955a} the sets $H_a$ only depend on the ordinal denoted by $a$, up to Turing equivalence.

\begin{theorem}[Clote \cite{clote1986generalization}]\label{CloteGenLimit}
There exists a partial computable function $g$ such that for all ordinal notations $a\in\mathcal{O}$, for all $e, x \in \Nat$
and $s \in \Nat^{<\om}$, $g(a,e,x,s)$ is defined and if $\varphi_e^{H_a}(x)= y$ then 
$\lim_{s \in B_a} g(a,e,x,s) = y.$
\end{theorem}

We will need the following corollary of the generalized limit lemma (it is proved as a Claim in the proof of Theorem 9 in \cite{Clote1984}).

\begin{lemma}[Clote, \cite{Clote1984}]\label{lem:clote}
For all $X\subseteq\Nat$ infinite, 
for all $a \in \mathcal{O}$,
for all $e,k \in \Nat$, there exists $s \in B_a\vert X$ such that for all $x\leq k$, 
$g(a,e,x,s)= \varphi_e^{H_a}(x)$, if $\varphi_e^{H_a}(x)$ is defined. 
\end{lemma}

From the proof in \cite{clote1986generalization} it is clear that in the above lemma the element $s\in B_a$ can be chosen so as to satisfy $k < \min(s)$. This will be useful in some later proofs. 

Clote \cite{Clote1984} used the above results to prove anti-basis results for Ramsey's theorem for barriers. With the same technique, we will obtain (weak) anti-basis results for the free set, thin set and rainbow Ramsey theorems 
for barriers. 

\subsection{Omission results}

We start with the barrier thin set theorem. 
The proof of the following theorem by Cholak et al. (Theorem 4.1 in \cite{Cholak_Giusto_Hirst_Jockusch_2005}) is a variation on Jockusch's proof of the analogous omitting theorem for Ramsey for pairs (Theorem 3.1 in \cite{Jockusch1972}).

\begin{theorem}[Cholak et al. \cite{Cholak_Giusto_Hirst_Jockusch_2005}]
There is a computable function $f:[\Nat]^2 \to \Nat$ such that no infinite $\Sigma^0_2$ set is thin for $f$. 
\end{theorem}

Using Clote's generalized limit lemma (Theorem \ref{CloteGenLimit} above) we obtain the following. 

\begin{theorem} \label{ThinNoCom}
For all $a \in \mathcal{O}$ there is a computable function $f: B_2 * B_a \to \Nat$ with no infinite thin set
computable in $H_a$. 
\end{theorem}

\begin{proof}
We want to define a computable function $f:B_2 * B_a \to \Nat$ such that for all $e,i\in\Nat$, if $X\subseteq\Nat$ is infinite and $X = \varphi_e^{H_a}$ then $i \in f([X]^{<\om} \cap (B_2*B_a))$, i.e. $(B_2*B_a)\vert X$ is not thin for $f$.

From Lemma \ref{CloteGenLimit} we have that for all $a \in \mathcal{O}$, for all 
$e, x \in \Nat$ and for all $s \in [\Nat]^{<\om}$, $g(a,e,x,s)$ is defined, $\{0,1\}$-valued and, moreover,
$$\varphi_e^{Ha}(x)\downarrow \,\Longrightarrow\, \lim_{s\in B_a} g(a,e,x,s)=1.$$
For each $e\in\Nat$, define $A_e$ as the set with characteristic function $\varphi_e^{H_a}$ if the latter is a characteristic function, and let $A_e$ be undefined otherwise. 
Denote by $F_{e,i}$ the set of least $\langle e,i\rangle+1$ elements of $A_e$, if $A_e$ is defined and has more than $\langle e,i\rangle$ elements. 
Let $F_{e,i,s}$ be the natural approximation of $F_{e,i}$ at stage $s$ where $s=(s_1, \ldots ,s_n)\in B_a$; that is if there are more than $\langle e,i\rangle $ numbers $x<s_1$ with $g(a,e,x,s) = 1$ then 
$F_{e,i,s}$ consists of the first $\langle e,i\rangle +1$ such numbers, otherwise let  $F_{e,i,s}$ be undefined.\smallskip

The construction of $f$ is carried out in stages $s=(s_1 \ldots ,s_n)\in B_a$ and $f(m,s)$ is defined at stage $s$ for each $m<s_1$. At stage $s=(s_1, \ldots, s_n) \in B_a$, with $s_1<\ldots <s_n$, there are substages $0,1,\ldots, \langle e,i\rangle, \ldots, s_1$ and $f$ is defined on at most one new argument at each substage.

\medskip
Substage $\langle e,i \rangle < s_1$ of stage $s$: if $F_{e,i,s}$ is not defined, continue to the next substage without doing anything. 
If $F_{e,i,s}$  is defined, let $m_{e,i,s}$ be the smallest element $m$ of $F_{e,i,s}$ with $f(m, s)$ not yet defined, and set $f(m_{e,i,s}, s):= i$.\smallskip

Substage $s_1$ of stage $s$: let $f(m,s):=1$ for all $m<s_1$, for which $f(m,s)$ is not yet defined. 
This completes the construction. 

\medskip
Now let $X\subseteq\Nat$ be an infinite, $H_a$-computable set. Let $\varphi_e^{H_a}$ be its characteristic function.
We show that $(B_2*B_a)\vert X$ is not thin for $f$.

If $X$ is as above then by Lemma \ref{lem:clote}, for all $i\in\Nat$, letting  $k=\max(F_{e,i})$,
there exists $ s \in B_a$ such that
$$(\forall x\leq k)(k < \min(s) \text{ and } g(a,e,x,s) = \varphi_e^{H_a}(x)).$$

By construction there exists $m\in X$ with $m<s_1$ where $f(m,s) = i$. Then $i \in f([X]^{<\om} \cap (B_2*B_a))$, thus $X $ is not thin for $f$ relative to the barrier $B_2*B_a$.
  \end{proof}

We now turn to the case of the barrier Ramsey theorem. We proceed as for the thin set theorem, using the proof of the following results by Csima and Mileti (\cite{CsimaMileti_2009}, Theorem 2.4) as a template. 

\begin{theorem}[Csima and Mileti]
There exists a computable $2$-bounded function $f:[\Nat]^2 \to \Nat$ such that no infinite $\Sigma^0_2$ set is a rainbow for $f$. 
\end{theorem}

We extend this result to barriers in the spirit of Clote \cite{Clote1984}.

\begin{theorem}\label{lb:rain}
For all $a \in \mathcal{O}$ there is a $2$-bounded computable $f: B_2 * B_a \to \Nat$ with no infinite rainbow
computable in $H_a$. 
\end{theorem}

\begin{proof}
We want to show that for all $e$, if $\varphi_e^{H_a}$ is the characteristic function of an an infinite subset $X$ of $b(B_a)$\footnote{ By definition it should be a subset of the base of $B_2 * B_a$ but the base of $B_2$ is $\Nat$.}
then $X$ is not a rainbow for $f$, in the sense that $f$ is not injective on $(B_2*B_a)\vert X$. 
To this aim it is sufficient to show that there are $\ell,m \in X$ and $s \in B_a$ such that $m < \ell < \min(s)$ for which $f(m, s) = f(\ell, s)$. 

Let $g$ be the partial computable function given by Theorem \ref{CloteGenLimit}.

We define a $2$-bounded computable $f:B_2 * B_a \to \Nat$ in stages, where the stages are elements of the barrier $B_a$. For each $s= (s_1,\dots, s_n)\in B_a$ 
at stage $s$ there are substages $0, \dots, e, \dots, s_1$. 

At substage $e < s_1$, if there exist $m < \ell < s_1$ such that $g(a,e,m,s) = g(a,e,\ell,s)=1$ which have not been claimed during stage $s$ then claim the least such $m$ and $\ell$ and set
$$ f(m,s) := \langle m, s\rangle =: f(\ell,s).$$ 
Otherwise, do nothing at this substage. Continue to the next substage.

At substage $s_1$ set $f(\ell,s) := \langle \ell, s\rangle$ for all $\ell < s_1$ for which $f(\ell,s)$ is not yet defined.

Notice that $f$ is computable and $2$-bounded by construction. The construction above is such that, for $s\in B_a$, the value $f(\ell,s)$ is defined
for all $\ell < s_1$ at stage $s$. Thus, overall, $f$ is defined on $B_2 * B_a$. 

Now let $X\subseteq\Nat$ be an infinite $H_a$-computable subset of $b(B_a)$ with characteristic function $\varphi_e^{H_a}$. 

From Theorem \ref{lem:clote} we can deduce the following.
Let $D_e$ denote the first $2e+2$ elements of $X$. 
Then there exists a $s \in B_a$ such that $\max(D_e) < \min(s)$ and for all $x \leq \max(D_e)$, we have 
$$ g(a,e,x,s) = \varphi_e^{H_a}(x).$$
We say that $s$ {\em is correct for} $D_e$.

By choice of $s$ and since $X$ is infinite there are exactly $2e+2$ numbers $x\leq \max(D_e) < \min(s)$ 
such that $g(a,e,x,s)=1$ (since $s$ is correct for $D_e$). Since at most $2e$ elements 
are claimed at stage $s$ before substage $e$, the construction at substage $e$ of stage $s$
claims the least unclaimed numbers $m, \ell < \min(s)$ with $g(a,e,m,s)= g(a,e,\ell,s)=1$ and defines $f(m,s) = f(\ell,s)$. 

By the minimality of $m$ and $\ell$ and by the correctness of $s$ for $D_e$
it must be the case that $m, \ell \in D_e \subseteq X$. 
But then $m, \ell \in X$ and $s \in B_a$ and $f(m,s)=f(\ell,s)$, 
so $X$ is not a rainbow for $f$ relative to the barrier $B_2*B_a$ (i.e. it is not the case that $f$ is injective on $(B_2*B_a)\vert X$).
\end{proof}

We briefly observe how some non-provability corollaries follow from the omitting theorems proved in the previous sections. Let $\alpha$ be a notation for a computable ordinal. Let $\Pi^0_\alpha$-$\mathsf{CA}_0$ be the system axiomatized by $\RCA_0$ augmented with 
the the well-ordering of $\alpha$ and the axiom of closure under the $\alpha$-th jump. Note that $\Pi^0_\alpha$-$\mathsf{CA}_0$ is strictly stronger than $\Pi^0_\beta$-$\mathsf{CA}_0$ if and only if $\alpha \geq \om\cdot \beta$. For a notation $a\in\mathcal{O}$ for an ordinal $\alpha\geq \om\cdot \beta$ the existence of a computable coloring of $B_2 * B_a$ with no thin set computable in ${H_a}$ (Theorem \ref{ThinNoCom}) implies that the barrier thin set theorem is not provable in $\Pi^0_\beta$-$\mathsf{CA}_0$, even if restricted to barriers of order type $\om^{1+\alpha}$. The same applies to the barrier free set theorem and to the barrier rainbow Ramsey theorem. 

\section{Conclusions and future work}

We formulated and proved natural generalizations of the free set theorem, the thin set theorem and the rainbow Ramsey theorem to colorings of barriers and we initiated the study of their effective and logical strength. 

We showed that the barrier free set theorem, the barrier thin set theorem and the barrier rainbow Ramsey theorem are consequences of the barrier Ramsey theorem. Moreover, the barrier thin set theorem and the barrier rainbow Ramsey theorem (for $2$-bounded colorings) are reducible to the barrier free set theorem as is the case in the fixed-dimension versions.  

We then showed that for each infinite computable ordinal $\alpha$, the free set, thin set and rainbow Ramsey theorems for barriers of order type $\om^{\alpha}$ have computable instances with no solution computable in $\emptyset^{(\alpha)}$. This was previously known only for $\alpha=\om$ (see \cite{carlucci2024ramseyliketheoremsschreierbarrier}). It remains to investigate whether computable instances of the free set, thin set or rainbow Ramsey theorems can code transfinite jumps. We plan to address this issue in future work.

\bibliographystyle{amsplain}
\bibliography{main-ref}

@article{carlucci_volpi_zdanowski_25,
	author = {Carlucci, Lorenzo and Volpi, Andrea and Zdanowski, Konrad},
	note = {arXiv:2505.02544},
	title = {The strength of {Ramsey's} theorem for coloring $\alpha$-large sets},
	year = {2026}}

@article{marcone_montalban_volpi_25,
	author = {Marcone, Alberto and Montalb{\'a}n, Antonio and Volpi, Andrea},
	note = {arXiv:2505.02544},
	title = {The barrier {Ramsey} theorem},
	year = {2025}}

@article{carroy_pequignot_2014,
	author = {Carroy, Rapha\"el and Pequignot, Yann},
	fjournal = {Fundamenta Mathematicae},
	journal = {Fundam. Math.},
	number = {3},
	pages = {247--270},
	title = {From well to better, the space of ideals},
	volume = {227},
	year = {2014}}

@article{Spector1955a,
	author = {Spector, Clifford},
	copyright = {https://www.cambridge.org/core/terms},
	doi = {10.2307/2266902},
	fjournal = {Journal of Symbolic Logic},
	issn = {0022-4812, 1943-5886},
	journal = {J. Symb. Log.},
	language = {en},
	month = mar,
	number = {2},
	pages = {151--163},
	title = {Recursive well-orderings},
	url = {https://www.cambridge.org/core/product/identifier/S0022481200096584/type/journal_article},
	urldate = {2025-06-10},
	volume = {20},
	year = {1955}}

@book{Rogers1987,
	address = {Cambridge, Mass},
	author = {Rogers, Hartley},
	edition = {1st MIT Press},
	isbn = {978-0-262-68052-3},
	publisher = {MIT Press},
	title = {Theory of recursive functions and effective computability},
	year = {1987}}

@article{cholak2020thin,
	author = {Cholak, Peter and Patey, Ludovic},
	doi = {10.1090/tran/7987},
	fjournal = {Transactions of the American Mathematical Society},
	issn = {0002-9947},
	journal = {Trans. Amer. Math. Soc.},
	mrclass = {03B30 (05A10)},
	mrnumber = {4069232},
	mrreviewer = {Huishan Wu},
	number = {4},
	pages = {2743--2773},
	title = {Thin set theorems and cone avoidance},
	url = {https://doi.org/10.1090/tran/7987},
	volume = {373},
	year = {2020}}

@article{wang2014some,
	author = {Wang, Wei},
	doi = {10.1016/j.aim.2014.05.003},
	fjournal = {Advances in Mathematics},
	issn = {0001-8708},
	journal = {Adv. Math.},
	mrclass = {03F35 (03B30)},
	mrnumber = {3213294},
	mrreviewer = {Alberto Marcone},
	pages = {1--25},
	title = {Some logically weak {R}amseyan theorems},
	url = {https://doi.org/10.1016/j.aim.2014.05.003},
	volume = {261},
	year = {2014}}

@article{carlucci2014strength,
	author = {Carlucci, Lorenzo and Zdanowski, Konrad},
	doi = {10.1017/jsl.2013.27},
	fjournal = {The Journal of Symbolic Logic},
	issn = {0022-4812},
	journal = {J. Symb. Log.},
	mrclass = {03F35 (03B30 05D10)},
	mrnumber = {3226013},
	mrreviewer = {Alberto Marcone},
	number = {1},
	pages = {89--102},
	title = {The strength of {R}amsey's theorem for coloring relatively large sets},
	url = {https://doi.org/10.1017/jsl.2013.27},
	volume = {79},
	year = {2014}}

@book{Todorcevic+2010,
	address = {Princeton},
	author = {Stevo Todorcevic},
	doi = {doi:10.1515/9781400835409},
	isbn = {9781400835409},
	lastchecked = {2024-11-14},
	publisher = {Princeton University Press},
	title = {Introduction to Ramsey Spaces},
	url = {https://doi.org/10.1515/9781400835409},
	year = {2010}}

@incollection{Cholak_Giusto_Hirst_Jockusch_2005,
	author = {Cholak, Peter A. and Giusto, Mariagnese and Hirst, Jeffry L. and Jockusch, Jr., Carl G.},
	booktitle = {Reverse mathematics 2001},
	mrclass = {03F35 (03B30 03D28 03D80)},
	mrnumber = {2185429},
	pages = {104--119},
	publisher = {Assoc. Symbol. Logic, La Jolla, CA},
	series = {Lect. Notes Log.},
	title = {Free sets and reverse mathematics},
	volume = {21},
	year = {2005}}

@article{Liu2022-LIUTRM,
	author = {Lu Liu and Ludovic Patey},
	doi = {10.1017/jsl.2021.98},
	fjournal = {Journal of Symbolic Logic},
	journal = {J. Symb. Log.},
	number = {1},
	pages = {313--346},
	title = {{The reverse mathematics of the Thin set and {E}rd\H{o}s-{M}oser theorems}},
	volume = {87},
	year = {2022}}

@article{CsimaMileti_2009,
	author = {Barbara F. Csima and Joseph R. Mileti},
	fjournal = {The Journal of Symbolic Logic},
	issn = {00224812, 19435886},
	journal = {J. Symb. Log.},
	number = {4},
	pages = {1310--1324},
	publisher = {[Association for Symbolic Logic, Cambridge University Press]},
	title = {The Strength of the Rainbow {Ramsey} Theorem},
	url = {http://www.jstor.org/stable/40378269},
	urldate = {2024-11-14},
	volume = {74},
	year = {2009}}

@incollection{marcone2005wqo,
	author = {Marcone, Alberto},
	booktitle = {Reverse mathematics 2001},
	mrclass = {03F35},
	mrnumber = {2185443},
	pages = {303--330},
	publisher = {Assoc. Symbol. Logic, La Jolla, CA},
	series = {Lect. Notes Log.},
	title = {Wqo and bqo theory in subsystems of second order arithmetic},
	volume = {21},
	year = {2005}}

@article{pouzet1972premeilleurordres,
	author = {Pouzet, Maurice},
	fjournal = {Annales de l'institut Fourier},
	journal = {Ann. Inst. Four.},
	number = {2},
	pages = {1--19},
	title = {Sur les pr{\'e}meilleurordres},
	volume = {22},
	year = {1972}}

@article{nash-williams1968better,
	author = {Nash-Williams, Crispin St. J. A.},
	doi = {10.1017/s030500410004281x},
	fjournal = {Proceedings of the Cambridge Philosophical Society},
	issn = {0008-1981},
	journal = {Proc. Cambridge Philos. Soc.},
	mrclass = {04.60 (05.00)},
	mrnumber = {221949},
	mrreviewer = {W. T. Tutte},
	pages = {273--290},
	title = {On better-quasi-ordering transfinite sequences},
	url = {https://doi.org/10.1017/s030500410004281x},
	volume = {64},
	year = {1968}}

@article{clote1986generalization,
	author = {Clote, Peter},
	doi = {10.2307/2274051},
	fjournal = {The Journal of Symbolic Logic},
	issn = {0022-4812},
	journal = {J. Symb. Log.},
	mrclass = {03D55 (03D25)},
	mrnumber = {840405},
	number = {2},
	pages = {273--291},
	title = {A Generalization of the Limit Lemma and Clopen Games},
	volume = {51},
	year = {1986}}

@article{galvinBorelSetsRamsey1973,
	author = {Galvin, Fred and Prikry, Karel},
	fjournal = {The Journal of Symbolic Logic},
	issn = {0022-4812},
	journal = {J. Symb. Log.},
	mrnumber = {0337630},
	pages = {193--198},
	title = {Borel Sets and {{Ramsey}}'s Theorem},
	volume = {38},
	year = {1973}}

@article{Pat:16,
	author = {Patey, Ludovic},
	doi = {10.1007/s11856-016-1433-3},
	fjournal = {Israel Journal of Mathematics},
	issn = {0021-2172},
	journal = {Israel J. Math.},
	mrclass = {03B30 (03F35)},
	mrnumber = {3557471},
	mrreviewer = {Fran\c{c}ois G. Dorais},
	number = {2},
	pages = {905--955},
	title = {The weakness of being cohesive, thin or free in reverse mathematics},
	url = {https://doi.org/10.1007/s11856-016-1433-3},
	volume = {216},
	year = {2016}}

@book{SIM:SOSOA,
	author = {Simpson, Stephen G.},
	doi = {10.1017/CBO9780511581007},
	edition = {Second},
	isbn = {978-0-521-88439-6},
	mrclass = {03F35 (03-02 03B30)},
	mrnumber = {2517689},
	pages = {xvi+444},
	publisher = {Cambridge University Press, Cambridge; Association for Symbolic Logic, Poughkeepsie, NY},
	series = {Perspectives in Logic},
	title = {Subsystems of second order arithmetic},
	url = {https://doi.org/10.1017/CBO9780511581007},
	year = {2009}}

@article{Jockusch1972,
	author = {Jockusch, Carl G.},
	doi = {10.2307/2272972},
	fjournal = {The Journal of Symbolic Logic},
	issn = {0022-4812, 1943-5886},
	journal = {J. Symb. Log.},
	language = {en},
	number = {2},
	pages = {268--280},
	title = {Ramsey's theorem and recursion theory},
	url = {https://www.cambridge.org/core/product/identifier/S0022481200079901/type/journal_article},
	urldate = {2023-02-15},
	volume = {37},
	year = {1972}}

@book{Dzhafarov.Mummert2022,
	address = {Cham},
	author = {Dzhafarov, Damir D. and Mummert, Carl},
	doi = {10.1007/978-3-031-11367-3},
	isbn = {978-3-031-11366-6 978-3-031-11367-3},
	language = {en},
	publisher = {Springer International Publishing},
	series = {Theory and {Applications} of {Computability}},
	shorttitle = {Reverse {Mathematics}},
	title = {Reverse {Mathematics}: {Problems}, {Reductions}, and {Proofs}},
	url = {https://link.springer.com/10.1007/978-3-031-11367-3},
	urldate = {2023-11-13},
	year = {2022}}

@article{Dorais.etal2015,
	author = {Dorais, Fran{\c c}ois and Dzhafarov, Damir and Hirst, Jeffry and Mileti, Joseph and Shafer, Paul},
	copyright = {https://www.ams.org/publications/copyright-and-permissions},
	doi = {10.1090/tran/6465},
	fjournal = {Transactions of the American Mathematical Society},
	issn = {0002-9947, 1088-6850},
	journal = {Trans. Amer. Math. Soc.},
	language = {en},
	number = {2},
	pages = {1321--1359},
	title = {On uniform relationships between combinatorial problems},
	url = {https://www.ams.org/tran/2016-368-02/S0002-9947-2015-06465-4/},
	urldate = {2025-05-12},
	volume = {368},
	year = {2015}}

@article{Clote1984,
	author = {Clote, Peter},
	copyright = {https://www.cambridge.org/core/terms},
	doi = {10.2307/2274171},
	fjournal = {Journal of Symbolic Logic},
	issn = {0022-4812, 1943-5886},
	journal = {J. Symb. Log.},
	language = {en},
	number = {2},
	pages = {376--400},
	title = {A recursion theoretic analysis of the clopen {Ramsey} theorem},
	url = {https://www.cambridge.org/core/product/identifier/S0022481200033442/type/journal_article},
	urldate = {2024-05-24},
	volume = {49},
	year = {1984}}

@article{carlucci2024ramseyliketheoremsschreierbarrier,
	author = {Lorenzo Carlucci and Oriola Gjetaj and Quentin Le Hou{\'e}rou and Ludovic Levy Patey},
	fjournal = {The Journal of Symbolic Logic},
	journal = {J. Symb. Log.},
	title = {Ramsey-like theorems for the {Schreier} barrier},
	year = {2025}}

@article{Assous,
	author = {Assous, Marc},
	journal = {Publications du D\'epartement de math\'ematiques (Lyon)},
	language = {fr},
	mrnumber = {366758},
	number = {4},
	pages = {89--106},
	publisher = {Universit\'e Claude Bernard - Lyon 1},
	title = {Caract\'erisation du type d'ordre des barri\`eres de {Nash-Williams}},
	url = {https://www.numdam.org/item/PDML_1974__11_4_89_0/},
	volume = {11},
	year = {1974},
	zbl = {0323.06001}}

\end{document}